\documentclass[reqno,A4paper]{amsart}
\usepackage{amsmath,amsthm,amssymb,amsfonts,amscd, array,latexsym}
\usepackage[noadjust]{cite}

\newtheorem{theorem}{Theorem}[section]
\newtheorem*{theorem*}{Theorem}

\newtheorem{corollary*}{Corollary}
\newtheorem{lemma}[theorem]{Lemma}
\newtheorem{proposition}[theorem]{Proposition}

\theoremstyle{definition}
\newtheorem{definition}[theorem]{Definition}
\newtheorem{example}[theorem]{Example}

\theoremstyle{remark}
\newtheorem{remark}[theorem]{Remark}

\numberwithin{equation}{section}

\makeatletter
\renewcommand{\theenumi}{\ensuremath{(\roman{enumi})}}

\renewcommand{\p@enumii}{}

\@namedef{subjclassname@2020}{%
  \textup{2020} Mathematics Subject Classification}
\makeatother

\newcommand{\vnorm}[1]{\left\|#1\right\|}

\newcommand{\acr}{\newline\indent}

\begin{document}
\title[Decomposition in direct sum]{Decomposition in direct sum of seminormed vector spaces and Mazur--Ulam theorem}

\author{Oleksiy Dovgoshey}
\address{\textbf{Oleksiy Dovgoshey}\acr
Department of Theory of Functions \acr
Institute of Applied Mathematics and Mechanics of NASU \acr
84100 Slovyansk, Ukraine, and\acr
Institut f\"{u}r Mathematik Universit\"{a}t zu L\"{u}beck\acr 
D-23562 L\"{u}beck, Deutschland}
\email{oleksiy.dovgoshey@gmail.com}

\author{J\"{u}rgen Prestin}
\address{\textbf{J\"{u}rgen Prestin}\acr
Institut f\"{u}r Mathematik Universit\"{a}t zu L\"{u}beck\acr 
D-23562 L\"{u}beck, Deutschland}
\email{prestin@math.uni-luebeck.de}

\author{Igor Shevchuk}
\address{\textbf{Igor Shevchuk}\acr
Faculty of Mechanics and Mathematics\acr
Tars Shevchenko National University of Kyiv\acr 
01601 Kyiv, Ukraine}
\email{shevchuk@univ.kiev.ua}



\begin{abstract}
It was proved by S. Mazur and S. Ulam in 1932 that every isometric surjection between normed real vector spaces is affine. We generalize the Mazur--Ulam theorem and find necessary and sufficient conditions under which distance-preserving mappings between seminormed real vector spaces are linear.
\end{abstract}

\maketitle

\section{Introduction}
\label{sec1}

The following result was obtained by S. Mazur and S. Ulam in \cite{MU1932CRASP}.

\begin{theorem}\label{t1.5}
Let \(X\) and \(Y\) be normed real vector spaces. Then every isometric bijection \(X \to Y\) is affine.
\end{theorem}

Short elegant proofs of this theorem can be found in \cite{Nica2012EM} and \cite{Vaei2003AMM}.

The main goal of the present paper is to generalize the Mazur--Ulam theorem to the case of seminormed real vector spaces. By term ``vector spaces'' we understand simultaneously the real vector spaces and the complex ones. If necessary we will indicate which of fields of scalars is used.

It should be noted here that, for the case of normed complex vector spaces, isometric bijections, generally speaking, are not affine, which is easy to see in the example of the bijection
\[
\mathbb{C} \ni z \mapsto \overline{z} \in \mathbb{C},
\]
where \(\mathbb{C}\) is the field of the complex numbers endowed by Euclidean norm and \(\overline{z}\) is the complex conjugate of \(z\).

Let us recall some basic definitions connected with seminormed vector spaces.

\begin{definition}\label{d2.1}
A \emph{seminorm} on a vector space \(X\) is a function \(\vnorm{\cdot} \colon X \to [0, \infty)\) such that \(\vnorm{x+y} \leqslant \vnorm{x} + \vnorm{y}\) and \(\vnorm{\alpha x} = |\alpha| \vnorm{x}\) for all \(x\), \(y \in X\) and every scalar \(\alpha\). A seminorm \(\vnorm{\cdot}\) is a \emph{norm} if \(\vnorm{x} = 0\) holds if and only if \(x=0_{X}\).
\end{definition}

In 1934 \DJ{}. Kurepa \cite{Kur1934CRASP} introduced the pseudometric spaces which, unlike metric spaces, allow the zero distance between different points. Let \(X\) be a set and let \(d \colon X \times X \to [0, \infty)\) be a symmetric function such that \(d(x, x) = 0\) for every \(x \in X\). The function \(d\) is a \emph{pseudometric} on \(X\) if it satisfies the triangle inequality
\[
d(x, y) \leqslant d(x, z) + d(z, y)
\]
for all \(x\), \(y\), \(z \in X\).

\begin{example}
If \((X, \vnorm{\cdot})\) is a seminormed vector space then the function
\begin{equation*}
\rho \colon X \times X \to [0, \infty), \quad \rho(x, y) = \vnorm{x - y},
\end{equation*}
is a pseudometric on \(X\).
\end{example}

The normed vector spaces can be characterized as seminormed vector spaces for which the corresponding pseudometrics are metrics.

\begin{definition}\label{d2.5}
Let \((X, d)\) and \((Y, \rho)\) be pseudometric spaces. A mapping \(\Phi \colon X \to Y\) is a \emph{pseudoisometric embedding} of \((X, d)\) in \((Y, \rho)\) if \(\rho(\Phi(x), \Phi(y)) = d(x, y)\) holds for all \(x\), \(y \in X\). If \(\Phi \colon X \to Y\) is a pseudoisometric embedding and, for every \(u \in Y\), there is \(v \in X\) such that \(\rho(\Phi(v), u) = 0\), then we say that \(\Phi\) is a \emph{pseudo-isometry}. The bijective pseudo-isometry \(\Phi \colon X \to Y\) is said to be an \emph{isometry} of \((X, d)\) and \((Y, \rho)\). Two pseudometric spaces are \emph{pseudoisometric} (\emph{isometric}) if there is a pseudo-isometry (isometry) of these spaces.
\end{definition}

For the case when pseudometric spaces \((X, d)\) and \((Y, \rho)\) are metric, we will say that a pseudoisometric embedding \(X \to Y\) is an isometric embedding of \((X, d)\) in \((Y, \rho)\).

The next lemma follows directly form Definition~\ref{d2.5}.

\begin{lemma}\label{l3.4}
Let \(U\), \(V\) and \(W\) be seminormed vector spaces. Then the following statements hold:
\begin{enumerate}
\item\label{l3.4:s1} A mapping \(U \xrightarrow{F} V \xrightarrow{\Phi} W\) is a pseudo-isometry whenever \(U \xrightarrow{F} V\) and \(V \xrightarrow{\Phi} W\) are pseudoisometries.
\item\label{l3.4:s2} If \(W\) is a normed vector space, then every pseudo-isometry \(V \xrightarrow{\Phi} W\) is surjective.
\end{enumerate}
\end{lemma}

If \((X, \vnorm{\cdot}_X)\) and \((Y, \vnorm{\cdot}_Y)\) are linear seminormed spaces and \(T \colon X \to Y\) is a pseudo-isometry of the pseudometric spaces \((X, d_X)\), \((Y, d_Y)\) endowed with pseudometrics \(d_X\) and \(d_Y\) generated by \(\vnorm{\cdot}_X\) and \(\vnorm{\cdot}_Y\), respectively, then we say that \(T\) is a pseudo-isometry of \((X, \vnorm{\cdot}_X)\) and \((Y, \vnorm{\cdot}_Y)\).

\begin{example}
Let \((X, \vnorm{\cdot}_X)\) and \((Y, \vnorm{\cdot}_Y)\) be seminormed vector spaces endowed by \emph{zero seminorms},
\[
\vnorm{x}_X = \vnorm{y}_Y = 0
\]
for all \(x \in X\) and \(y \in Y\), and let \(T \colon X \to Y\) be an arbitrary mapping. Then the following statements hold:
\begin{enumerate}\renewcommand{\theenumi}{\ensuremath{(\arabic{enumi})}}
\item \(T\) is a pseudoisometric embedding.
\item \(T\) is a pseudo-isometry.
\item \(T\) is an isometry iff \(T\) is bijective.
\item \(T\) is a linear pseudo-isometry iff \(T\) is a linear.
\end{enumerate}
\end{example}

\begin{remark}\label{r1.12}
The concept of isometry of metric spaces can be extended to pseudometric spaces in various non-equivalent ways. J.~Kelley \cite{Kelley1975} defined the isometries of pseudometric spaces \(X\) and \(Y\) as the distance-preserving surjections \(X \to Y\). It is clear that every isometry in Kelley's sense is a pseudo-isometry in the sense of Definition~\ref{d2.5} and, moreover, every isometry in the sense of Definition~\ref{d2.5} is also Kelley's isometry. Another generalization of isometries are the combinatorial similarities of pseudometric spaces (see \cite{DLAMH2020, DovBBMSSS2020, BD2022BKMSip, BD2023arXiv}). 
\end{remark}

The paper is organized as follows. 

Section~\ref{sec2} contains auxiliary results connected with pseudoisometries of seminormed vector spaces and their decompositions in a direct sum of some special subspaces.

The main new result of the paper is Theorem~\ref{t3.2} which gives us the necessary and sufficient conditions under which a pseudometric embedding \(X \to Y\) of seminormed real vector spaces \(X\) and \(Y\) is linear. In Proposition~\ref{p3.2} and Theorem~\ref{t3.5} we consider the cases when one from the spaces \(X\) and \(Y\) is a normed real vector space. 

In the final Section~\ref{sec4} we formulate an open problem connected with linearity of pseudoisometric embeddings.

\section{Auxiliary results}
\label{sec2}

Let \(Y\) and \(Z\) be linear subspaces of a vector space \(X\) and let
\begin{equation}\label{s2:e2}
Y \cap Z = \{0_X\}.
\end{equation}
Then the set 
\begin{equation}\label{s3:e3}
Y \oplus Z := \{y + z \colon y \in Y \text{ and } z \in Z\}
\end{equation}
will be called the \emph{direct sum} of \(Y\) and \(Z\). It is clear that \(Y \oplus Z\) is a linear subspace of \(X\).

\begin{theorem}\label{t2.2}
If \(X\) is a vector space and \(Z\) is a linear subspace of \(X\), then there exists a linear subspace \(Y\) of \(X\) such that \(X = Y \oplus Z\).
\end{theorem}

For the proof see, for example, Theorem~03 in \cite[p.~21]{PR1968}.

\begin{example}\label{ex2.2}
Using the triangle inequality and the homogeneity of the seminorms, it is easy to prove that, for every seminormed vector space \((X, \vnorm{\cdot})\), the set
\begin{equation}\label{ex2.2:e1}
Z_X := \{z \in X \colon \vnorm{z} = 0\}
\end{equation}
is a linear subspace of \(X\) (see Theorem~1.34 in \cite[p.~26]{Rud1991}). Consequently, by Theorem~\ref{t2.2}, there is a linear subspace \(Y\) of \(X\) such that \(X = Y \oplus Z_X\).
\end{example}

The following lemma is also closely related to Theorem~1.34 of \cite{Rud1991}.

\begin{lemma}\label{l3.11}
Let \((X, \vnorm{\cdot}_X)\) be a seminormed vector space and let \(Y\) be a linear subspace of \(X\) which satisfies the equality
\begin{equation}\label{l3.11:e1}
X = Y \oplus Z_X,
\end{equation}
where \(Z_X\) is defined by \eqref{ex2.2:e1}. Write \(\vnorm{\cdot}_Y\) for the restriction of the seminorm \(\vnorm{\cdot}_X\) on the set \(Y\). Then \((Y, \vnorm{\cdot}_Y)\) is normed and linearly pseudoisometric to \((X, \vnorm{\cdot}_X)\).
\end{lemma}

\begin{proof}
Equality \eqref{l3.11:e1} implies that, for every \(x \in X\), there exist unique \(y = y(x) \in Y\) and \(z = z(x) \in Z_X\) for which
\begin{equation}\label{l3.11:e2}
x = y + z
\end{equation}
holds. Consequently, there is a unique \(\Phi \colon X \to Y\) such that, for every \(x \in X\), we have
\begin{equation}\label{l3.11:e3}
x = \Phi(x) + z,
\end{equation}
where \(z = z(x)\) is the same as in \eqref{l3.11:e2}. Using \eqref{s2:e2} and \eqref{s3:e3}, we see that \(\Phi\) is a linear map. 

Let \(z \in Z_X\) and let \(x \in X\). Then we have 
\[
\|x+z\| \leqslant \|x\| + \|z\| = \|x\| = \|x + z - z\| \leqslant \|x+z\| + \|z\| = \|x + z\|
\]
and, hence, \(\|x + z\| = \|x\|\) for all \(x \in X\) and \(z \in Z_X\). Now it follows from~\eqref{l3.11:e3} and \eqref{ex2.2:e1} that
\[
\vnorm{\Phi(x_1) - \Phi(x_2)}_Y = \vnorm{x_1 - x_2}_X
\]
holds for all \(x_1\), \(x_2 \in X\). Moreover, the uniqueness of \(y = y(x)\) and \(z = z(x)\) in \eqref{l3.11:e2} implies the equality \(\Phi(y) = y\) for each \(y \in Y\). Thus, \(\Phi \colon X \to Y\) is a linear pseudo-isometry of \((X, \vnorm{\cdot}_X)\) and \((Y, \vnorm{\cdot}_Y)\). To complete the proof it suffices to show that the seminorm \(\vnorm{\cdot}_Y\) is a norm.

The space \(Y\) is a normed subspace of \(X\) if and only if 
\begin{equation}\label{p2.3:e2}
Z_{Y} = \{0_X\}.
\end{equation}
Since \(Y\) is a linear subspace of \(X\), we have the inclusion \(Z_{Y} \subseteq Z_X\). Now~\eqref{p2.3:e2} follows from \eqref{l3.11:e1} and the definition of \(Y \oplus Z_X\). Hence, \(Y\) is a normed subspace of \(X\).
\end{proof}

\section{Mazur--Ulam theorem for seminormed vector spaces}
\label{sec3}

Let us start from a suitable for us reformulation of the Mazur--Ulam theorem.

\begin{theorem}\label{t1.9}
Let \(\Phi \colon X \to Y\) be an isometric embedding of normed real vector spaces \(X\) and \(Y\). Then the following statements are equivalent:
\begin{enumerate}
\item \label{t1.9:s1} \(\Phi\) is linear.
\item \label{t1.9:s2} The set \(\Phi(X)\) is a linear subspace of \(Y\) and the equality \(\Phi(0_X) = 0_Y\) holds.
\end{enumerate}
\end{theorem}

\begin{proof}
The implication \(\ref{t1.9:s1} \Rightarrow \ref{t1.9:s2}\) is evidently valid. Let \ref{t1.9:s2} hold. Then \(\Phi\) is affine by the Mazur--Ulam theorem. Now it suffices to note that the affine mapping \(\Phi\) is linear iff \(\Phi(0_X) = 0_Y\) holds.
\end{proof}

The following example shows that, for a seminormed real vector space \(Y\), a pseudoisometric embedding \(\Phi \colon X \to Y\) satisfying the equality \(\Phi(0_X) = 0_Y\) can be non-linear even if \(X\) is a normed real space.

\begin{example}\label{ex3.1}
Let \((\mathbb{R}^2, \vnorm{\cdot})\) be a two-dimensional seminormed real vector space endowed with the seminorm \(\vnorm{\cdot}\) such that
\[
\vnorm{(x, y)} := |x|
\]
for each \((x, y) \in \mathbb{R}^2\). Let us consider the subspace \(X\) of \(\mathbb{R}^2\) defined by
\[
X := \{(x, y) \in \mathbb{R}^2 \colon y = 0\}.
\]
Then \(X\) is a normed one-dimensional linear subspace of \((\mathbb{R}^2, \vnorm{\cdot})\) and \(\mathbb{R}^2 = X \oplus Z_{\mathbb{R}^2}\) holds. Let \(f \colon \mathbb{R} \to \mathbb{R}\) be an arbitrary function such that \(f(0) = 0\). Then the mapping \(F \colon X \to \mathbb{R}^2\),
\[
F(x, y) = (x, 0) + (0, f(x))
\]
is a pseudo-isometry of \(X\) and \(\mathbb{R}^2\), and the equality \(F(0_{\mathbb{R}^2}) = 0_{\mathbb{R}^2}\) holds.
\end{example}

The following proposition shows that Theorem~\ref{t1.9} remains true if the image space \(Y\) is only seminormed vector space. This is possible because the pseudometric embedding transfers the norm property of \(X\) onto \(Y\).

\begin{proposition}\label{p3.2}
Let \((X, \vnorm{\cdot}_X)\) and \((Y, \vnorm{\cdot}_Y)\) be seminormed real vector spaces and let \(\Phi \colon X \to Y\) be a pseudoisometric embedding. If \(X\) is normed, then the following statements are equivalent:
\begin{enumerate}
\item \label{p3.2:s1} \(\Phi\) is a linear mapping.
\item \label{p3.2:s2} \(\Phi(X)\) is a linear subspace of \(Y\) and the equality 
\begin{equation}\label{p3.2:e1}
\Phi(0_X) = 0_Y
\end{equation}
holds.
\end{enumerate}
\end{proposition}

\begin{proof}
\(\ref{p3.2:s1} \Rightarrow \ref{p3.2:s2}\). If \(\Phi\) is linear, then \ref{p3.2:s2} follows from the definition of linear mappings.

\(\ref{p3.2:s2} \Rightarrow \ref{p3.2:s1}\). Let \ref{p3.2:s2} hold and let the seminorm \(\vnorm{\cdot}_X\) be a norm. Let us consider two different points \(x_1\), \(x_2 \in X\). Since \((X, \vnorm{\cdot}_X)\) is normed and \(\Phi\) is a pseudoisometric embedding, we have
\[
0 < \vnorm{x_1 - x_2}_X = \vnorm{\Phi(x_1)-\Phi(x_2)}_Y.
\]
Consequently, \(\Phi\) is injective. It implies \(\vnorm{y}_Y > 0\) for every 
\[
y \in Y \setminus \{0_Y\}.
\]
Indeed, if \(y \in \Phi(X)\) and \(y \neq 0_Y\), then we can find \(x \in X \setminus \{0_X\}\) such that \(y = \Phi(x)\). Now using \eqref{p3.2:e1} we obtain
\[
0 < \vnorm{x}_X = \vnorm{x - 0_X}_X = \vnorm{\Phi(x) - \Phi(0_X)}_Y = \vnorm{y - 0_Y}_Y = \vnorm{y}_Y.
\]
Hence, the restriction of the seminorm \(\vnorm{\cdot}_Y\) on the vector space \(\Phi(X)\) is a norm. Thus, \(\Phi(X)\) is a normed vector subspace of \(Y\). Consequently, \(\Phi\) is linear by Theorem~\ref{t1.9}.
\end{proof}

Thus, if additionally \(Y\) is only a seminormed vector space, we just have to be able to guarantee the linearity on \(Z_X\).

\begin{theorem}\label{t3.2}
Let \((X, \vnorm{\cdot}_X)\) and \((Y, \vnorm{\cdot}_Y)\) be seminormed real vector spaces and let \(\Phi \colon X \to Y\) be a pseudoisometric embedding. Then the following statements are equivalent:
\begin{enumerate}
\item \label{t3.2:s1} \(\Phi\) is a linear mapping.
\item \label{t3.2:s2} The restriction \(\Phi|_{Z_X}\) is linear and there is a linear subspace \(W_X\) of \(X\) such that
\begin{equation}\label{t3.2:e1}
X = W_X \oplus Z_X,
\end{equation}
and \(\Phi(W_X)\) is a linear subspace of \(X\), and
\begin{equation}\label{t3.2:e2}
\Phi(w+z) = \Phi(w) + \Phi(z)
\end{equation}
for all \(w \in W_X\) and \(z \in Z_X\).
\end{enumerate}
\end{theorem}

\begin{proof}
\(\ref{t3.2:s1} \Rightarrow \ref{t3.2:s2}\). Let \(\Phi\) be a linear mapping. As was noted in Example~\ref{ex2.2}, the set \(Z_X\) is a linear subspace of \(X\). Consequently, \(\Phi|_{Z_X}\) is also linear as a restriction of a linear mapping. By Theorem~\ref{t2.2}, there is a linear subspace \(W_X\) of \(X\) such that \eqref{t3.2:e1} holds. Since \(\Phi\) is linear, \(\Phi(W_X)\) is a linear subspace of \(Y\). The linearity of \(\Phi\) also implies \eqref{t3.2:e2} for all \(w \in W_X\) and \(z \in Z_X\).

\(\ref{t3.2:s2} \Rightarrow \ref{t3.2:s1}\). Let \ref{t3.2:s2} hold. Let us denote by \(W_X\) a linear subspace of \(X\) satisfying \eqref{t3.2:e1} such that \(\Phi(W_X)\) is a linear subspace of \(Y\) and \eqref{t3.2:e2} holds for all \(w \in W_X\) and \(z \in Z_X\). 

To prove \ref{t3.2:s1} it suffices to show that 
\begin{equation}\label{t3.2:e5}
\Phi(\alpha_1 x_1 + \alpha_2 x_2) = \alpha_1 \Phi(x_1) + \alpha_2 \Phi(x_2)
\end{equation}
holds for all \(x_1\), \(x_2 \in X\) and all scalars \(\alpha_1\), \(\alpha_2\). To do so, using decomposition \eqref{t3.2:e1}, we can find \(z_i \in Z_X\) and \(w_i \in W_X\) such that 
\[
x_i = z_i + w_i, \quad i=1, 2.
\]
Consequently, we have 
\[
\alpha_1 x_1 + \alpha_2 x_2 = (\alpha_1 z_1 + \alpha_2 z_2) + (\alpha_1 w_1 + \alpha_2 w_2)
\]
with \(\alpha_1 z_1 + \alpha_2 z_2 \in Z_X\) and \(\alpha_1 w_1 + \alpha_2 w_2 \in W_X\). Statement~\ref{t3.2:s2} and Proposition~\ref{p3.2} imply that
\begin{equation}\label{t3.2:e6}
\begin{aligned}
\Phi(\alpha_1 x_1 + \alpha_2 x_2) & =\Phi((\alpha_1 z_1 + \alpha_2 z_2) + (\alpha_1 w_1 + \alpha_2 w_2)) \\
& = \Phi(\alpha_1 z_1 + \alpha_2 z_2) + \Phi(\alpha_1 w_1 + \alpha_2 w_2)\\
& = \Phi|_{Z_X}(\alpha_1 z_1 + \alpha_2 z_2) + \Phi|_{W_X}(\alpha_1 w_1 + \alpha_2 w_2).
\end{aligned}
\end{equation}
By Lemma~\ref{l3.11}, \(W_X\) must be a normed vector space and, together with the assumption \(\Phi(0_X) = 0_Y\), the linearity of \(\Phi\) on \(W_X\) follows from Proposition~\ref{p3.2}. Hence, equation~\eqref{t3.2:e5} is satisfied, thus, \(\Phi\) is linear on the whole domain \(X\).
\end{proof}

The final result of the paper, Theorem~\ref{t3.5}, is similar to Proposition~\ref{p3.2}. In Proposition~\ref{p3.2} \(X\) was normed and now \(Y\) is normed. Moreover, \(\Phi\) is now even a pseudo-isometry and was only a pseudometric embedding. This stronger assumption is needed in order for \(\Phi\) to transfer the normed vector space property of \(Y\) onto \(X\), so that again Theorem~\ref{t1.9} can be applied.

\begin{theorem}\label{t3.5}
Let \(X\) and \(Y\) be seminormed real vector spaces and let \(\Phi \colon X \to Y\) be a pseudo-isometry. If \(Y\) is normed, then the following statements are equivalent:
\begin{enumerate}
\item \label{t3.5:s1} \(\Phi\) is a linear mapping.
\item \label{t3.5:s2} The equality
\begin{equation}\label{t3.5:e1}
\Phi(0_X) = 0_Y
\end{equation}
holds.
\end{enumerate}
\end{theorem}

\begin{proof}
Let \(Y\) be a normed vector space. 

The implication \(\ref{t3.5:s1} \Rightarrow \ref{t3.5:s2}\) is trivially valid.

Let us prove the validity of \(\ref{t3.5:s2} \Rightarrow \ref{t3.5:s1}\). Suppose that \eqref{t3.5:e1} holds. By Theorem~\ref{t3.2}, statement \ref{t3.5:s1} holds if \(\Phi|_{Z_X}\) is linear, and there is a linear subspace \(W_X\) of \(X\) such that 
\begin{equation}\label{t3.5:e2}
X = W_X \oplus Z_X,
\end{equation}
and \(\Phi(W_X)\) is a linear subspace, and \eqref{t3.2:e2} holds for all \(w \in W_X\) and \(z \in Z_X\).

Let us consider an arbitrary linear subspace \(W_X\) of \(X\) satisfying \eqref{t3.5:e2}. The identical embedding \(I_W \colon W_X \to X\), 
\[
I_W(w) = w \text{ for every } w \in W_X,
\]
is a pseudo-isometry of \(W_X\) and \(X\). Consequently, the mapping
\[
W_X \xrightarrow{I_{W_X}} X \xrightarrow{\Phi} Y
\]
is also a pseudo-isometry by Lemma~\ref{l3.4}. Now, using Lemma~\ref{l3.4} again, we see that this pseudo-isometry is a surjection. In particular, \(\Phi(W_X)\) is a linear subspace of \(Y\), because \(\Phi(W_X) = Y\) holds.

Equality~\eqref{t3.5:e1} and Definition~\ref{d2.5} give us the inclusion
\begin{equation}\label{t3.5:e3}
\Phi(Z_X) \subseteq Z_Y.
\end{equation}
Since \(Y\) is a normed vector space, we have \(Z_Y = \{0_Y\}\), that, together with~\eqref{t3.5:e3}, implies the linearity of \(\Phi(Z_X)\).

To complete the proof  it suffices to show that 
\begin{equation}\label{t3.5:e4}
\Phi(w + z) = \Phi(w) + \Phi(z)
\end{equation}
holds for all \(w \in W_X\) and \(z \in Z_X\). To do it we note that \eqref{t3.5:e4} can be written in the form 
\begin{equation}\label{t3.5:e5}
\Phi(w + z) = \Phi(w),
\end{equation}
because \(\Phi(Z_X) \subseteq Z_Y = \{0_Y\}\). Since \(Y\) is a normed vector space, \eqref{t3.5:e5} holds if and only if 
\begin{equation}\label{t3.5:e6}
\vnorm{\Phi(w + z) - \Phi(w)}_Y = 0.
\end{equation}
By equality~\eqref{ex2.2:e1} and Definition~\ref{d2.5}, we have
\[
0 = \vnorm{z}_X = \vnorm{(w+z)-w}_X = \vnorm{\Phi(w + z) - \Phi(w)}_Y,
\]
which implies \eqref{t3.5:e5}.
\end{proof}

\section{Strict convexity}
\label{sec4}

Recall that a normed vector space \((X, \vnorm{\cdot})\) is \emph{strictly convex} if, for all \(x\), \(y \in X\), the equality 
\begin{equation*}
\vnorm{x+y} = \vnorm{x} + \vnorm{y}
\end{equation*}
implies \(x = ty\) for some \(t \geqslant 0\). 

It was proved by J. A. Baker \cite{Bak1971AMM} that a normed real vector space \(Y\) is strictly convex if and only if, for every normed real vector space \(X\), each isometric embedding \(X \to Y\) is affine. (For other characterizations of the strictly convex normed vector spaces see \cite{GS1976PAMS} and references therein.)

It seems to be interesting to generalize the concept of strict convexity and to find an analog of Baker's result for seminormed vector spaces.

\section*{Funding}

Oleksiy Dovgoshey was supported by Volkswagen Stiftung Project ``From Modeling and Analysis to Approximation''.


\end{document}